\documentclass[12pt]{amsart}
\usepackage[utf8]{inputenc}
\usepackage[T1]{fontenc}

\usepackage[margin=3cm]{geometry}

\usepackage{amsthm, amsmath, amssymb}

\usepackage{tcolorbox}

\usepackage{hyperref}
\usepackage{graphicx}

\hypersetup{
  pdftitle = {Multiparty equality in the local broadcast model},
  pdfauthor = {},
  colorlinks = true,
  linkcolor = black!30!red,
  citecolor = black!30!green
}

\newtheorem{theorem}{Theorem}[section]

\newtheorem{lemma}[theorem]{Lemma}
\newtheorem*{lemma*}{Lemma}
\newtheorem{corollary}[theorem]{Corollary}

\theoremstyle{remark}
\newtheorem{remark}{Remark}[section]

\newcommand{\Cbn}{\mathsf{Cov}_{\rm{bnd}}}
\newcommand{\Cbl}{\mathsf{Cov}_{\rm{balls}}}
\newcommand{\Pbn}{\mathsf{Pack}_{\rm{bnd}}}
\newcommand{\Pbl}{\mathsf{Pack}_{\rm{balls}}}
\newcommand{\tbn}{\tau_{\rm{bnd}}}
\newcommand{\tbl}{\tau_{\rm{balls}}}
\newcommand{\nbn}{\nu_{\rm{bnd}}}
\newcommand{\nbl}{\nu_{\rm{balls}}}
\newcommand{\tvc}{\mathsf{tvc}}

\newcommand{\vc}{\mathsf{vc}}
\newcommand{\wds}{\mathsf{wds}}

\DeclareMathOperator{\opt}{\mathsf{OPT}}
\newcommand{\cost}{\mathsf{cost}}

\newcommand{\intv}[2]{\left \{ #1, \dots, #2 \right \}}

\title{Multiparty equality in the local broadcast model}
\author[L.~Esperet]{Louis Esperet}
\address[L.~Esperet]{CNRS, Université Grenoble Alpes, G-SCOP,
  Grenoble, France}
\email{louis.esperet@grenoble-inp.fr}

\author[J.-F.~Raymond]{Jean-Florent Raymond}
\address[J.-F.~Raymond]{CNRS, ENS de Lyon, Université Claude Bernard Lyon 1, LIP, UMR 5668, Lyon, France}
\email{jean-florent.raymond@cnrs.fr}

\date{\today}

\thanks{The authors are partially supported by the French ANR Projects ENEDISC
  (ANR-24-CE48-7768) and GRALMECO (ANR-21-CE48-0004), and by LabEx
  PERSYVAL-lab (ANR-11-LABX-0025).}

\begin{document}

\begin{abstract}
    In this paper we consider the \emph{multiparty equality problem} in graphs, where every vertex of a graph $G$ is given an input, and the goal of the vertices is to decide whether all inputs are equal. We study this problem in the \emph{local broadcast model}, where a message sent by a vertex is received by all its neighbors and the total cost of a protocol is the sum of the lengths of the messages sent by the vertices.
    This setting was studied by Khan and Vaidya, who gave in 2021 a protocol achieving a 4-approximation in the general case.
    
    We study this multiparty communication problem through the lens of network topology. 
    We design a new protocol for 2-connected graphs, whose efficiency relies on the notion of total vertex cover in graph theory. This protocol outperforms the aforementioned 4-approximation in a number of cases. To demonstrate its applicability, we apply it to obtain optimal or asymptotically optimal protocols for several natural network topologies such as cycles, hypercubes, and grids. On the way we also provide new bounds of independent interest on the size of total vertex covers in regular graphs.
\end{abstract}

\maketitle

\section{Introduction}

We consider the \emph{multiparty equality problem} in graphs, where every vertex of a graph $G$ is given a $k$-bit input, and the goal of the vertices is to decide whether all inputs are the same. Each vertex can communicate with its neighbors in the graph, and the cost of the protocol is the total number of bits transmitted during the communication phase. After this phase, each vertex either accepts or rejects the instance; if every vertex was assigned the same input, then all vertices must accept the instance, and if two of the inputs differ then at least one vertex must reject the instance.  The \emph{complexity} of the multiparty equality problem is the minimum cost of a protocol solving the problem. Note that we only consider protocols that are deterministic and \emph{static} \cite{alon2017testing,khan2021testing,LV11}, in the sense that the set of vertices sending messages and the size of their messages is independent of their input strings (they only depend on $G$ and $k$). One the other hand, the content of the messages is allowed to depend on the input strings. See \cite{alon2017testing} for a discussion on non-static protocols.

\smallskip

In the classical \emph{point-to-point} communication model, each vertex communicates with its neighbors on different channels (one channel per vertex), so that sending the same message of $\ell$ bits to $d$ neighbors costs $d\ell$ bits of communication. A natural lower bound on the complexity of the multiparty equality problem in this model can be obtained by observing that at least $k$ bits need to be sent through every cut of the graph in every protocol. As a consequence, as proved in \cite{CR15,alon2017testing}, the complexity is at least $k$ times the minimum fractional transversal of cuts in the graph (or by duality, the maximum fractional packing of cuts), see Section~\ref{sec:lp} for more details. Alon, Efremenko and Sudakov \cite{alon2017testing} showed that this linear programming lower bound can be attained asymptotically for a number of graph classes, including Hamiltonian graphs, and moreover they gave a protocol of cost within a factor $4/3$ of the optimal for every graph. To the best of our knowledge, it is unknown whether computing the optimal cost of a protocol can be done in polynomial time. 

\smallskip

Much less is known in the \emph{local broadcast} communication model, which is the main focus of this paper. In this model, each message sent by a vertex $v$ is transmitted to \emph{all} the neighbors of $v$, and the number of bits of this message is counted only once in the complexity of the protocol. The \emph{total cost} of a protocol $\Pi$ on $G$, denoted $\cost_\Pi(G,k)$, is the sum of the number of bits broadcast by each vertex. The \emph{total complexity} (minimum cost of a protocol solving the problem) of the multiparty equality problem in the local broadcast communication model  on $G$ with $k$ bit inputs is denoted by $\opt(G,k)$. We will be mostly interested in the \emph{per-bit complexity} $\opt(G)=\lim_{k\to \infty} \tfrac1k \opt(G,k)$, whose existence directly follows from Fekete's subadditive lemma \cite{Fekete}. Similarly, we define the \emph{per-bit cost} of a protocol $\Pi$ as $\cost_{\Pi}(G)=\lim_{k\to \infty} \tfrac1k \cost_{\Pi}(G,k)$ (this limit does not exist for all protocols, but it does in all the protocols we consider in the paper). 
As in the point-to-point communication model, a simple linear programming lower bound on the complexity of the problem can be obtained by noting that for every cut, the sum of the number of bits broadcast by the vertices incident to the cut has to be at least $k$. This implies that a natural lower bound on $\opt(G)$ is the maximum fractional packing (or minimum fractional transversal) of boundaries in the graph $G$, where a \emph{boundary} is the set of vertices incident to the edges of a given cut in the graph. 

\smallskip

In the definition of a protocol for equality above, we have required that when two inputs differ, this is detected by at least one vertex. We note that when focusing on the per-bit complexity $\opt(G)$ or the per-bit cost of a protocol, it would be equivalent to ask the stronger requirement that at then end of the protocol, all vertices know whether they have the same input or not. This is because if two inputs differ, then some vertex will detect it, and this vertex can then communicate this information to all vertices along a spanning tree, costing a number of bits that depends only on $G$, not on the length $k$ of the input. The contribution of this  additional constant cost in the per-bit complexity  vanishes as $k\to \infty$.

\smallskip

The complexity of multiparty equality in the local broadcast model was studied by Khan and Vaidya  \cite{khan2021testing}. 
They first studied so-called \emph{simple protocols}, where each vertex either broadcasts its entire input, or remains inactive. Each vertex then checks that the inputs received from its neighbors match its own input, and accepts the instance if and only if this is the case. They related such protocols to the notion of a weakly connected dominating set. A \emph{weakly connected dominating set} $S$ in a graph  $G$, is a subset of the vertices of $G$ such that the set of edges incident to $S$ induces a spanning and connected subgraph of $G$. The minimum size of a weakly connected dominating set in $G$ is denoted by $\wds(G)$. 

\smallskip

Khan and Vaidya   proved the following characterization of simple protocols \cite{khan2021testing}. 

\begin{theorem}[\cite{khan2021testing}]\label{thm:kvsimple}
A simple protocol solves the multiparty equality problem if and only if the set $S \subseteq V$ of vertices chosen to transmit their entire input is a weakly connected dominating set of $G$. In particular, $\opt(G)\le \wds(G)$.
\end{theorem}

Let us denote by $\Pi_0$ the simple protocol of Theorem~\ref{thm:kvsimple}, of per-bit cost $\cost_{\Pi_0}(G)=\wds(G)$. The authors of \cite{khan2021testing}  use the result above to provide examples where the multiplicative gap between the minimum cost of a simple protocol and $\opt$ is of order $\Omega(\log n)$ (here and  in the remainder of the paper, $n$ stands for the number of vertices of the graph under consideration). Khan and Vaidya then studied general (non-simple) protocols, and designed a protocol $\Pi_1$ which is within factor 4 of the optimal.

\begin{theorem}[\cite{khan2021testing}]\label{thm:kv}
For any graph $G$ and integer $k$, there is a protocol $\Pi_1$ for multiparty equality in the local broadcast model whose total cost is at most $4k$ times the maximum fractional packing of boundaries, and in particular at most $4\opt(G,k)$.
\end{theorem}


We note that it appears to be unknown whether $\opt(G)$ and $\opt(G,k)$ can be computed in polynomial time.

\subsection*{Our results} In this paper, we study the multiparty equality problem in the local broadcast model through the lens of network topology (or equivalently, graph classes).

\smallskip

We first provide a couple of interesting applications of  Theorem~\ref{thm:kvsimple}: an optimal protocol for trees, and an asymptotically optimal protocol for hypercubes.

\smallskip

We then design a new protocol in 2-connected graphs, which outperforms Theorem~\ref{thm:kv} on a number of natural graph classes. Our protocol uses ideas introduced by Alon, Efremenko and Sudakov \cite{alon2017testing} in the point-to-point model, relying on classical constructions in extremal combinatorics. Recall that a \emph{vertex cover} in a graph $G$ is a subset of vertices intersecting every edge of $G$, and a \emph{total dominating set} in $G$ is a subset $X$ of vertices such that every vertex in $G$ has a neighbor in $X$. The efficiency of our protocol depends on the minimum size of a \emph{total vertex cover} in the graph $G$, denoted by $\tvc(G)$, where a total vertex cover is a subset of vertices which is both a vertex cover and a total dominating set. We prove the following result.

\begin{theorem}\label{thm:tvc}
For any 2-connected graph $G$, there is a protocol $\Pi_2$ for multiparty equality in the local broadcast model whose per-bit cost is at most $\tfrac12 \, \tvc(G)$, and in particular $\opt(G)\le \tfrac12\, \tvc(G)$.
\end{theorem}

We emphasize that only the vertices of a total vertex cover send messages in the protocol $\Pi_2$, as opposed to the protocol of \cite{alon2017testing} whose correctness relies on the fact that every vertex communicates. 

\medskip

Theorem~\ref{thm:tvc} motivates the study of total vertex covers in graphs, for which we obtain new results which might be of independent interest. 
Our first application concerns the class of cycles. The linear programming lower bound on $\opt(C_n)$ is of order $n/3$ (assigning weight $\tfrac13$ to each vertex). On the other hand, as $\wds(C_n)=\lfloor n/2\rfloor$, the simple protocol of Theorem~\ref{thm:kvsimple} has per-bit cost $\lfloor n/2\rfloor$, and it can be checked that the protocol of Theorem~\ref{thm:kv} has per-bit cost $\tfrac23n-\tfrac23$ or $\tfrac23n-\tfrac13$, depending on the parity of $n$. 
As an immediate application of Theorem~\ref{thm:tvc}, we obtain an almost optimal protocol for cycles.

\begin{corollary} For any integer $n\ge 3$, $\tfrac{n}3\le \opt(C_n)\le \tfrac{n+1}3$, and for every integer $n\equiv 0 \pmod 3$,  $\opt(C_n)=\tfrac{n}3$.
\end{corollary}

We then consider hypercubes, a well-studied graph topology. It is known that $d$-dimensional hypercubes $Q_d$ have weakly connected dominating sets of size matching asymptotically the linear programming lower bound $\frac{2^{d}}{d+1}$ \cite{Gri21}, and therefore the protocol $\Pi_0$ of Theorem~\ref{thm:kvsimple} is asymptotically best possible in this class. However the bounds on the size of weakly connected dominating sets depend on non-trivial results on the density of primes (even in the most simple case $d=2^\ell-1$). We give a simple and self-contained proof of an upper bound on $\tvc(Q_d)$ in this case, which directly implies that $\Pi_2$ is asymptotically optimal for these graphs as well.

\begin{corollary}\label{cor:hyper}
Let $\ell\ge 2$ be an integer and let $d=2^\ell-1$. Then the protocol $\Pi_2$ for equality in the hypercube $Q_d$ has per-bit cost at most \[\frac{2^{d}}{d+1}+\frac{2^{d/2}}{d+1}=\opt(Q_d)+O\left(\sqrt{\opt(Q_d)}\right).\]
\end{corollary}

We next consider the classical $n$ by $n$ square grid, which is denoted by $G_n$. The linear programming lower bound on $\opt(G_n)$ is of order $n^2/5$ (assigning weight $\tfrac15$ to each vertex), but a simple counting argument shows that $\wds(G_n)\ge n^2/4$, and thus simple protocols cannot be optimal. The protocol of Theorem~\ref{thm:kv} has many possible outcomes in a grid, depending on the order in which the vertices are chosen, and therefore it is not clear if it can be optimal for some well-chosen ordering. We have found some natural orderings for which the protocol of Theorem~\ref{thm:kv} has per-bit cost $\tfrac25\,n^2$, that is twice the linear programming lower bound. On the other hand, as a direct application of Theorem~\ref{thm:tvc}, we obtain an asymptotically optimal protocol for grids.

\begin{corollary} $ \opt(G_n)= (\tfrac{1}5+o(1))n^2$.
\end{corollary}

We then study $d$-regular graphs for $d\ge 3$. An interesting property of these classes is that although the classes are very large (containing roughly $2^{(d/2-1)n\log n}$ non-isomorphic $n$-vertex graphs), the linear programming lower bound in this case can be expressed easily as $\tfrac{n}{d+1}$, which allows to compare explicitly the quality of protocols on a large class of graphs by simply bounding the cost of protocols on the entire class (instead of comparing the cost of the protocol and the linear programming lower bound on every single graph in the class). 

\smallskip

It was proved in \cite{chen2017note} that every cubic graph on $n$ vertices has a total vertex cover on at most $3n/4$ vertices. We extend this result to every degree by proving that for $d\ge 3$,  every $d$-regular graph on $n$ vertices has a total vertex cover on at most $\tfrac{d}{d+1}\cdot n$ vertices. This is optimal for every $d$, since any total vertex cover of the complete graph $K_{d+1}$ (which is $d$-regular) contains at least $d$ vertices. It was proved in \cite{chen2017note} that there are two more extremal connected graphs in the case $d=3$, and all the other connected cubic graphs have a total vertex cover with less than $3n/4$ vertices. Our result below implies that for $d\ge 5$, complete graphs are the only extremal examples. 

\begin{theorem}\label{thm:tvcregular}
For every $d\ge 3$, every $d$-regular graph on $n$ vertices has a total vertex cover on at most $\frac{d}{d+1} \cdot n$ vertices. For $d\ge 5$, every connected $d$-regular graph on $n$ vertices distinct from $K_{d+1}$ has a total vertex cover on at most $\tfrac{d-\epsilon}{d+1}\cdot n$ vertices, with $\epsilon= \tfrac1{2d+1}$.
\end{theorem}

Combining Theorem~\ref{thm:tvcregular} and Theorem~\ref{thm:tvc} for the upper bound, and using the linear programming lower bound of $\tfrac{n}{d+1}$ alluded to above, we immediately obtain the following.

\begin{corollary}\label{cor:d}
For any $d\ge 3$ and  any 2-connected $d$-regular graph  $G$, the protocol $\Pi_2$ of Theorem~\ref{thm:tvc} has per-bit cost at most $\frac{d}{2d+2}\cdot n = \tfrac{d}2\cdot \opt(G)$. If moreover,  $d\ge 5$ and $G\ne K_{d+1}$, then $\Pi_2$ has per-bit cost at most $\tfrac{d-\epsilon}2 \cdot \opt(G)$, for some $\epsilon>0$ depending only on $d$.
\end{corollary}

Corollary~\ref{cor:d} improves on the 4-approximation of Theorem~\ref{thm:kv} for any degree $d\le 8$.

\subsection*{Organization of the paper}
In Section~\ref{sec:prel} we introduce the necessary notation, the linear programming lower bounds and two useful lemmas.
Section~\ref{sec:simple} is devoted to the study of simple protocols. In Section~\ref{sec:tvc} we introduce a new protocol for 2-connected graphs based on total vertex covers and provide several applications. Directions for future research are given in Section~\ref{sec:ccl}.

\section{Preliminaries}\label{sec:prel}

\subsection{Notation}
In this paper, $\log$ denotes the binary logarithm.
Let $G$ be a graph.
For two disjoint sets $S,S' \subseteq V(G)$, we denote by $E(S,S')$ the set of
edges of $G$ with the one endpoint in $S$ and the other one in $S'$, and the \emph{cut} defined by $S$ is $E(S,\bar{S})$, where $\bar{S}=V(G)\setminus S$.
The \emph{boundary} $B(S)$ of the cut defined by $S$ consists of all vertices
incident to $E(S, \bar{S})$.
Observe that for every $v\in V(G)$, the trivial boundary $B(\{v\})$  is precisely 
the closed neighborhood $N[v]$ of $v$ in $G$ (that is, $v$ together with its set of neighbors in $G$).

\smallskip

In the remainder we only consider connected graphs (otherwise the problem we consider has no solution, as vertices in different components cannot communicate to verify that their inputs are equal).
A graph $G$ is \emph{2-connected} if it has at least 3 vertices and $G$ remains connected after the removal of any vertex. 

\subsection{A linear programming lower bound}\label{sec:lp}

Consider the following covering linear program for boundaries.
\begin{tcolorbox}[title=\textbf{Program} $\Cbn(G)$]
  variables: $x(v)$ for every $v\in V(G)$

  function to minimize: $\sum\limits_{v\in V(G)} x(v)$
  \medskip
  
  subject to the constraints: $\left \{
  \begin{array}{l}
  \forall S\subsetneq V(G)\ \text{s.t.}\ S\neq\emptyset,\ \sum\limits_{v\in B(S)} x(v) \geq 1\\
    \forall v \in V(G),\ x(v) \geq 0.
\end{array}\right .$
\medskip

\textbf{Optimal value:} $\tbn^*(G)$
\end{tcolorbox}

The dual packing linear program is the following.

\begin{tcolorbox}[title=\textbf{Program} $\Pbn(G)$]
variables: $y(S)$ for every non-empty $S \subsetneq V(G)$
\medskip

maximize: $\sum\limits_{S\subsetneq V(G),\ S\ne \emptyset} y(S)$
\medskip

subject to: $\left \{
\begin{array}{l}
    \forall v\in V(G),\ \sum\limits_{S,\ v\in B(S)} y(S) \leq 1\\
    \forall S \subsetneq V(G)\ \text{s.t.}\ S\neq \emptyset,\ y(S) \geq 0 
\end{array}
\right .$
\medskip

\textbf{Optimal value:} $\nbn^*(G)$
\end{tcolorbox}

\medskip

Note that by linear programming duality, $\tbn^*(G)=\nbn^*(G)$ for any graph $G$. 
It will be also useful to consider the simple variants of the two linear programs above, $\Cbl(G)$  and $\Pbl(G)$, where instead of considering all boundaries $B(S)$, we only consider the trivial boundaries (of the form $B(\{v\})$, for some $v$). Let $\tbl^*(G)=\nbl^*(G)$ be the associated optimal values.

Observing that for any non-empty cut $E(S,\bar{S})$, at least $k$ bits in total have to be broadcast by the vertices incident to the cut, the following was proved in \cite{khan2021testing}.

\begin{theorem}[\cite{khan2021testing}]\label{thm:lb}
For any graph $G$ and integer $k$, $\opt(G,k)\ge \tbn^*(G) \cdot k$, and in particular $\opt(G)\ge \tbn^*(G)\ge \tbl^*(G) $.
\end{theorem}

We also consider the integer parameters $\tbn(G)$, $\nbn(G)$, $\tbl(G)$, $\nbl(G)$, which are defined similarly as their fractional counterparts, but optimizing over the integers instead of the real (or rational) numbers. For instance, as $B(\{v\})=N[v]$ for every vertex $v$, $\tbl(G)$ is equal to the minimum size of a set $S$, such that $V(G)=\bigcup_{v\in S}N[v]$, or equivalently $\tbl(G)$ is the \emph{domination number} of $G$, the minimum size of a dominating set in $G$. Therefore, $\tbl^*(G)$ can be considered as the fractional domination number of $G$. The following simple result will be useful in the remainder of the paper.

\begin{lemma}\label{lem:deg}
If $G$ is an $n$-vertex graph with maximum degree $\Delta$, then $\tbn^*(G)\ge \tbl^*(G)\ge \tfrac{n}{\Delta+1}$. Moreover, if $G$ is $\Delta$-regular, then $\tbl^*(G)=\tfrac{n}{\Delta+1}$.
\end{lemma}

\begin{proof}
Setting $y(S)=\tfrac1{\Delta+1}$ for every vertex boundary $S$ of the form $S=B(\{v\})=N[v]$ for some $v\in V(G)$, we obtain a feasible solution to the program $\Pbl(G)$, and thus  $\tbl^*(G)=\nbl^*(G)\ge \tfrac{n}{\Delta+1}$. When $G$ is $\Delta$-regular, setting $x(v)=\tfrac1{\Delta+1}$ for every vertex $v\in V(G)$, we obtain a feasible solution to the program $\Cbl(G)$. It follows that $\tbl^*(G)\le \tfrac{n}{\Delta+1}$, and thus $\tbl^*(G)=\tfrac{n}{\Delta+1}$, as desired.
\end{proof}

We note that we will only use the first part of the statement, in combination with Theorem~\ref{thm:lb}.

\medskip

Finally, we will use the following well-known fact about 2-connected graphs.

\begin{lemma}\label{lem:2co}
Every 2-connected graph $G$ on $n$ vertices has a 2-connected spanning subgraph $H$ on at most $2n-3\le 2n$ edges. 
\end{lemma}

\begin{proof}
Every 2-connected graph $G$ has an ear-decomposition such that every ear is open (the two endpoints of the ear are distinct), see \cite{Whi32}. Include in $H$ all the ears from the ear-decomposition, except those consisting of a single edge. Note that we start with a cycle (of length at least 3), and every time we add $k$ new vertices to $H$, we add $k+1\le 2k$ edges to $H$. Thus $H$ has at most $2(n-3)+3=2n-3$ edges.
\end{proof}

\section{Simple protocols}\label{sec:simple}

We recall that in \emph{simple protocols} for equality, each vertex either broadcasts its entire input, or does not send any bit of communication. As stated in Theorem~\ref{thm:kvsimple}, the set $S$ of vertices broadcasting their input in such a protocol is a \emph{weakly connected dominating set}, which implies that $\opt(G)\le \wds(G)$ for any graph $G$. In this section, we explore a number of interesting consequences of this result.

\medskip

Let $\vc(G)$ denote the minimum size of a \emph{vertex cover} of a graph $G$ (a set of vertices intersecting every edge of $G$). Note that $\wds(G)\le \vc(G)$. We obtain the following simple result.

\begin{theorem}\label{thm:trees}
For any tree $T$, the simple protocol in which every vertex of a vertex cover broadcasts its input is optimal among all protocols. In particular $\opt(T)=\vc(T)$.
\end{theorem}

\begin{proof}
 Let $T$ be a tree. Since the minimal cuts of $T$ are single edges, $\tbn^*(T)$ coincides with the optimal solution of the linear relaxation of vertex cover in $T$. But since $T$ is bipartite, it follows from K\H{o}nig's theorem that this optimal solution is equal to $\vc(T)$. As there is a protocol for equality with per-bit cost $\wds(G)\le \vc(T)$, the per-bit complexity of equality in trees is precisely the vertex cover number.
 \end{proof}
 
 As observed in the previous section, $\tbl(G)$ is the domination number of $G$, and this can be used to connect $\tbl(G)$ and $\wds(G)$ as follows.
 
 \begin{lemma}\label{lem:wdstbl}
For any connected graph $G$, $\tbl(G) \leq \wds(G)\le 2\, \tbl(G)-1$.
\end{lemma}

\begin{proof}
Since a weakly dominating set is a dominating set, we have $\tbl(G) \leq \wds(G)$, so it remains to prove $\wds(G)\le 2\, \tbl(G)$.
Consider a dominating set $D$ of size $\tbl(G)$ in $G$. We assume that $G$ has more than one vertex otherwise the statement is trivially true. Let $H$ be the subgraph of $G$ induced by all edges that are incident to a vertex of $D$. Since $D$ is a dominating set, $H$ is a spanning subgraph of $G$ of minimum degree at least 1, and with at most $\tbl(G)$ connected components. Let $E'$ be a set of edges of $E(G)\setminus E(H)$ of minimum size such that the subgraph of $G$ induced by the edges of $E(H)\cup E'$ is connected. As $H$ has at most $\tbl(G)$ connected components, $|E'| \leq \tbl(G)-1$. Let $D'$ be a subset of vertices of $G$ obtained by doing the following for each edge $e\in E'$: select an endpoint of $e$ arbitrarily and add it to $D'$. Then $D\cup D'$ has size at most $2\,\tbl(G)-1$ and it can be checked that $D\cup D'$ is a weakly dominating set of $G$, as desired.
\end{proof}

Recall that $\tbl^*(G) \leq \opt(G)$.
In particular, for any graph $G$ for which $\tbl(G)$ is very close to $\tbl^*(G)$, combining Theorem~\ref{thm:kvsimple}  with Lemma~\ref{lem:wdstbl} immediately provides a good approximation of the optimal broadcast protocol.

\begin{corollary}\label{cor:simple}
For every graph $G$, there is a simple protocol for equality with per-bit cost at most $2\cdot \tfrac{\tbl(G)}{\tbl^*(G)}\cdot \opt(G)$.
\end{corollary}

For instance, the $d$-dimensional hypercube $Q_d$ with $d=2^k-1$, for some integer $k$, satisfies $\tbl^*(Q_d)=\tbl(Q_d)=\tfrac{2^d}{d+1}$, and thus Corollary~\ref{cor:simple} implies that the broadcast protocol based on weakly dominating sets is within a multiplicative factor 2 of the optimal protocol for these graphs.

\medskip

Note that for hypercubes, we can use the stronger results obtained by Griggs in \cite{Gri21}, showing that $Q_d$ has a connected dominating set of size $(1+o(1))\tfrac{2^d}{d+1}$ as $d\to \infty$, and in particular $\wds(Q_d)=(1+o(1))\tfrac{2^d}{d+1}$. Theorem~\ref{thm:kvsimple} then gives an asymptotically optimal (simple) protocol for equality in hypercubes. Alternatively, we describe in the next section a protocol $\Pi_2$ which has per-bit cost $\tfrac{2^d}{d+1}+ 2^{d/2}$ on the hypercube $Q_d$, see Section~\ref{ssec:hypc} for details.

\section{Total vertex covers and a new protocol}\label{sec:tvc}

\subsection{The new protocol}

Recall that a \emph{total vertex cover} in a graph $G$ is a vertex subset that is both a total dominating set and a vertex cover. The minimum size of such a set is denoted by $\tvc(G)$. See Figures \ref{fig:grid} and \ref{fig:grid2} for illustrations of total vertex covers in grids. In this section we prove Theorem~\ref{thm:tvc}, restated below as Theorem~\ref{thm:tvc2}.

\begin{theorem}\label{thm:tvc2}
Let $G$ be a 2-connected graph $G$ and  $k$ be an integer. Then there is a protocol $\Pi_2$ for equality in $G$ of total cost $\cost_{\Pi_2}(G,k)\le (k/2 +o(k))  \cdot \tvc(G)$, and thus  $\opt(G)\le \cost_{\Pi_2}(G) \le \tfrac12\, \tvc(G)$.
\end{theorem}

We emphasize that in our protocol, only the vertices of a total vertex cover of $G$ send a message; all the other vertices remain silent.

\smallskip

Note that it might be the case that a 2-connected graph $G$ contains a 2-connected spanning subgraph $H$ with $\tvc(H)<\tvc(G)$, in which case it is natural to run the equality protocol in $H$ rather than in $G$. We therefore obtain  the following immediate corollary of Theorem~\ref{thm:tvc2}, which will be used extensively in our applications.

\begin{corollary}\label{cor:tvc}
For any 2-connected spanning subgraph $H$ of $G$, there is a protocol for equality  in the local broadcast model in $G$ whose per-bit cost is at most $\tfrac12\, \tvc(H)$, and in particular $\opt(G)\le \tfrac12\, \tvc(H)$.
\end{corollary}

\begin{remark}\label{rem:tvc}
In terms of transversals, a total vertex cover is an (integral) transversal of open neighborhoods (total dominating set) and of edges (vertex cover). As we are dealing with graphs with no isolated vertices, a total vertex cover always exists, for instance every vertex cover of such a graph $G$ is dominating and can be made total by adding a neighbor of each of its elements, so $\tvc(G) \leq 2\vc(G)$.
\end{remark}

Our protocol $\Pi_2$ of Theorem~\ref{thm:tvc2} relies on a technical lemma (Lemma~\ref{lem:alon} below) from \cite{alon2017testing} and used there in the different setting of point-to-point communication. Before we can formally state this result, we need to introduce some terminology.
Let $H$ be a graph with $n$ vertices $v_1, \dots, v_n$. Let $F$ be a graph with its vertices partitioned into $n$ classes $U_1, \dots, U_n$. A subgraph of $F$ isomorphic to $H$ is called a \emph{special copy} of $H$ if for every $i \in \intv{1}{k}$ the vertex corresponding to $v_i$ in the copy belongs to $U_i$. We say that $F$ is a \emph{faithful host} for $H$ if the two following conditions are met:
\begin{enumerate}
    \item The edges of $F$ can be partitioned into edge-disjoint special copies of $H$; and
    \item $F$ contains no other special copy of $H$ than the aforementioned $|E(F)|/|E(H)|$ copies defining its edge set.
\end{enumerate}
See Figure \ref{fig:faithful} for an illustration (the three special copies of the triangle in the faithful host are depicted with different colors and width).

\begin{figure}[htb]
  \centering
    \includegraphics[scale=1.2]{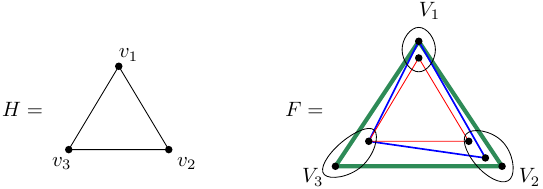}
  \caption{A faithful host of the triangle.\label{fig:faithful}}
\end{figure}

\begin{lemma}[{\cite[Lemma 3.3]{alon2017testing}}]\label{lem:alon}
Let $H$ be a 2-connected $n$-vertex graph and let $m$ be a positive integer. Then there is a faithful host for $H$ with classes of vertices $U_1, \dots, U_n$, each of size $nm$, containing a least
$m^2/ e^{10\sqrt{\log m \log n}}
$
special copies of $H$.
\end{lemma}

We are now ready to prove the main result of this section.
\begin{proof}[Proof of Theorem~\ref{thm:tvc}]
  Let $n$ denote the number of vertices of $G$, that we call $v_1,
  \dots, v_n$, and let $m$ be the minimum integer such that
  $\frac{m^2}{e^{10\sqrt{\log m \log n}}} \geq 2^k$.
  By Lemma~\ref{lem:alon} there is a faithful host $F$ for $G$ with classes of vertices $U_1, \dots, U_n$, each of size $nm$, containing a least $2^k$ special copies of $G$.
  For every $i \in \intv{1}{n}$, let us name
  $u_i^1, \dots, u_i^{nm}$ the vertices of $U_i$. 
 We associate to each $k$-bit word $w$ a special copy $G_w$ of $G$ in $F$ in an injective way (i.e. if $w\neq w'$ then $G_w \neq G_{w'}$). Since there are at least $2^k$ special copies, this is always possible.

  For a fixed assignment $\lambda \colon V(G) \to \{0,1\}^k$ of $k$-bit inputs to the vertices of $G$, let us define the \emph{identity} of a vertex $v_i$ of $G$ as the integer
  $1\le j\le nm$ such that the vertex of the special copy $G_{\lambda(v_i)}$ corresponding to $v_i$ is $u_i^j$.
  In other words,
  $V(G_{\lambda(v_i)}) \cap U_i = \{u_i^j\}$. Such an integer always exists by the definition
  of a special copy.
  Notice that the identity of a vertex $v_i$ depends on its input $\lambda(v_i)$.
  
  \smallskip

  For every edge $v_iv_j \in E(G)$, we say that the identity $a_i$ of $v_i$ is \emph{consistent} with the identity $a_j$ of $v_j$ if:
  \begin{enumerate}
  \item there is an edge in $F$ between $u^{a_i}_i$ and $u_j^{a_j}$;
    and
  \item this edge belongs to the special copy $G_{\lambda(v_j)}$.
  \end{enumerate}
  Observe that due to the second item, this relation is not symmetric. However, given $a_i$, the vertex $v_j$ can check whether $a_i$ is consistent with $a_j$.
  
  \medskip
  
  We now fix an optimal total vertex cover $S$ of $G$, and we are ready to describe the protocol $\Pi_2$.
  Given an assignment $\lambda$ of $k$-bit inputs to the vertices of
  $G$, the protocol is the following:
  
  \smallskip
  \begin{tcolorbox}[title=\textbf{Protocol}]  
    \noindent\textit{Communication:} Each vertex $v$ of $S$ broadcasts its
    identity to its neighbors and the other vertices remain silent.
    
    \smallskip
    
    \noindent \textit{Decision:} A vertex $v$ of $G$ accepts the instance if and only if all identities received from its neighbors are consistent with its own identity.
\end{tcolorbox}
    \smallskip
  
  This completes the description of the protocol $\Pi_2$. Recall that we chose $F$ and the function $w\mapsto G_w$ depending only on $G$, and in particular they do not depend on the assignment of inputs $\lambda$. So using them and their own inputs, the vertices can indeed compute their own identity and check consistency with those possibly sent by their neighbors.
  
  \smallskip

  Let us now show that the protocol is correct.
  Suppose first that $\lambda$ assigns the same word $w$ to every vertex of $G$. Then for every edge $vv'$ of $G$, the identity of $v$ is trivially consistent with that of $v'$, since these are defined with respect to the same special copy $G_w$ of $G$. Hence in this case every vertex accepts the instance.
  
  Conversely, suppose that every vertex accepts the instance. For every $i \in \intv{1}{n}$, let $a_i$ be the identity of $v_i$.
  Recall that $S$ is a vertex cover, so for every edge $v_iv_j$ of $G$, at least one of the two endpoints lies in $S$, say $v_i\in S$. It follows that $v_i$ sends its identity $a_i$ to $v_j$. As $v_j$ accepts the instance, $u_i^{a_i}u_j^{a_j}$ is an edge in $F$. Hence $F$ has a subgraph $G'$ on vertex set $\{u_i^{a_i}\}_{i\in \intv{1}{n}}$ that is isomorphic to $G$ and where for every $i\in \intv{1}{n}$ the vertex corresponding to $v_i$ lies in the vertex subset $U_i$. The subgraph $G'$ of $F$ is a thus a special copy of $G$ in $F$, and by the definition of a faithful host $G'$ is one of the special copies that partition the edges of $F$, say $G'=G_w$ for some $k$-bit word $w$. 
  Since $S$ is a total vertex cover, every vertex $v_j$ of $G$ has a neighbor $v_i\in S$. By the second item in the definition of consistency, the edge $v_i^{a_i}v_j^{a_j}$ of $F$ belongs to the special copy $G_{\lambda(v_j)}$. As noted above this edge also belongs to $G'=G_w$ so by definition of faithful hosts, $\lambda(v_j) = w$. This shows that all vertices of $G$ have the same input, as desired.
  
  \smallskip
  
  It remains to bound the cost of the protocol. Recall that only the vertices in $S$ broadcast their identity. Also, recall that the identity of a vertex is an integer of $\intv{1}{nm}$, and that $n$ (the number of vertices of $G$) is a constant while $k\to \infty$.
  By the choice of $m$ we have
  \begin{align*}
  2\log (m-1) &\leq k + 10\log(e)\sqrt{\log(m-1)\log n}\\
  \log m &= \frac{k}{2} + O(\sqrt{\log m \log n})\\
  &= \frac{k}{2} + O(\sqrt{k \log n})
  \end{align*}
  The total cost of the protocol is thus at most
  \begin{align*}
      |S|\cdot \log mn &= |S| \cdot \left (\frac{k}{2} + O\left (\sqrt{k \log n} \right ) + \log n \right)\\
      &= \tvc(G)\cdot  \left (\frac{k}{2} + o(k) \right ).
  \end{align*}
It follows that the per-bit cost of the protocol is at most  $\tfrac12\,\tvc(G)$, which completes the proof.
\end{proof}

\subsection{Applications}

\subsubsection{Complete bipartite graphs $K_{2,t}$}\label{sec:biclique}
We start by giving a simple example of a family of graphs where the protocol $\Pi_2$ of Theorem~\ref{thm:tvc2} outperforms the protocols $\Pi_0$ of Theorem~\ref{thm:kvsimple} and $\Pi_1$ of Theorem~\ref{thm:kv}.
The complete bipartite graph $K_{2,t}$ has a total vertex cover that consists of three vertices: the two vertices of degree $t$ and one vertex of degree 2.
According to Theorem~\ref{thm:tvc2}, we thus have $\cost_{\Pi_2}(K_{2,t}) \leq 3/2$. This is tight as we also have $\opt(K_{2,t})\geq 3/2$ thanks to the lower bound $\opt(K_{2,t}) \geq \tbn^*(K_{2,t})$ (Theorem~\ref{thm:lb}).
On the other hand, we can observe (see \cite{khan2021testing}) that on $K_{2,t}$ the cost of the protocols $\Pi_0$ and $\Pi_1$ is always at least 2.

\subsubsection{Cycles} We immediately deduce from Theorem~\ref{thm:tvc2} an almost optimal protocol for equality in cycles. 

\begin{corollary}\label{cor:cycles}
On input $C_n$, the protocol $\Pi_2$ of Theorem~\ref{thm:tvc2} has per-bit cost at most
$\tfrac{n+1}{3}\le \opt(C_n) + \frac{1}{3}$. Moreover, if $n\equiv 0 \pmod 3$, $\Pi_2$ has per-bit cost $\tfrac{n}3=\opt(C_n)$.
\end{corollary}
\begin{proof}
In a cycle $C_n$ we can construct a total vertex cover by selecting all vertices except those whose index is 0 modulo 3.
So $\tvc(C_n) \leq\frac{2}{3}(n+1)$ and the protocol of Theorem~\ref{thm:tvc2} has per-bit cost at most $(n+1)/3$ (and at most $n/3$ if $n\equiv 0 \pmod 3$).

 On the other hand, by Lemma~\ref{lem:deg}, $\tbl^*(C_n) = n/3$.
 Overall we get:
 \[
   \frac{n}{3} = \tbl^*(C_n) \leq \tbn^*(C_n) \leq \frac{1}{k} \opt(C_n,k) \leq \cost_{\Pi_2}(C_n)\leq \frac{n+1}{3},
 \]
as desired.
\end{proof}

\subsubsection{Hypercubes}\label{ssec:hypc}

In Section~\ref{sec:simple}, we observed that there is a simple protocol for equality in the $d$-dimensional hypercube $Q_d$ of per-bit cost $(1+o(1))\tfrac{2^d}{d+1}$, asymptotically matching the fractional lower bound $\tfrac{2^d}{d+1}$. The efficiency of this protocol was based on a recent result of Griggs on connected dominating sets in hypercubes \cite{Gri21}. We note that this result itself is based  on classical results on $q$-ary codes and crucially relies on the density of primes, even for simple cases such as $d=2^\ell-1$ (for some integer $\ell$) where the domination number of $Q_d$ is well understood. It turns out that Corollary~\ref{cor:tvc} can be used to give an alternative protocol for equality in the hypercube $Q_d$, $d=2^\ell-1$, of cost at most $\tfrac{1}{d+1}(2^d+2^{d/2})$, which only relies on basic arguments (once we assume Lemma~\ref{lem:alon}, which is based on a non-trivial construction of dense sets of integers without long arithmetic progression \cite{Beh46}).

\begin{theorem}\label{thm:hypercube}
Let $\ell\ge 2$ be an integer and let $d=2^\ell-1$. Then the hypercube $Q_d$ has a spanning 2-connected subgraph $H$ with $\tvc(H)\le 2\cdot (2^{d-\ell}+2^{d/2-\ell})$.
\end{theorem}

\begin{proof}
We start by recalling a number of classical properties of Hamming codes.
Let $d=2^\ell-1$  and let $H$ be an $\ell$ by $d$ binary matrix  whose column vectors $h_1, \ldots, h_{d}$ are all the non-zero vectors in $\textrm{GF(2)}^\ell$. We assume for convenience that $h_1=\mathbf{1}$ (the all 1 vector). Let $C_0$ be the subgroup of $\textrm{GF(2)}^d$ consisting of all vectors $y$ such that $Hy=0$. For every $1\le i \le d$, we set $C_i=C_0+e_i$, where $e_i$ denotes the vector of $\textrm{GF(2)}^d$ whose entries are all 0 except at coordinate $i$ (the vectors $e_i$, $1\le i \le d$, form the standard basis of $\textrm{GF(2)}^d$ if viewed as a vector space, and  a generating set if  viewed as an additive group). We view the sets $C_i$ both as subsets of the vertex set $V(Q_d)$ and as cosets of the subgroup $C_0$ of $\textrm{GF(2)}^d$. In particular the sets $C_i$, $0\le i\le d$, partition $V(Q_d)$ and all have cardinality $\tfrac{2^d}{d+1}$. The crucial property of this construction is that for any $x\in C_i$ and any $j\ne i$, $x$ has exactly one neighbor $y$ in $C_j$: $y=x+e_s$, where $s$ is the index of the column $h_i+h_j$ in $H$. Note that the translation vector $e_s$ only depends on $i$ and $j$, so there is indeed a perfect matching between $C_i$ and $C_j$ in $Q_d$.

\smallskip

We now consider $S=C_0\cup C_1$ (which we view both as a subset of vertices of $Q_d$ and as a subgroup of $\textrm{GF(2)}^d$). Let $G_S$ be the subgraph of $Q_d$ induced by all edges having at least one endpoint in $S$. Note that $G_S$ is a spanning subgraph of $Q_d$: all vertices of $S$ have degree $d$ in $S$, and all vertices not in $S$ have degree 2 (being adjacent to exactly one vertex of $C_0$ and one vertex of $C_1$, by the paragraph above). Moreover all the vertices of $S$ are in the same orbit under the action of the automorphism group of $G_S$, so no vertex of $S$ is a cut-vertex in $G_S$ (otherwise all of them would be cut-vertices), and thus  all connected components of $G_S$ are  2-connected.

\medskip

\noindent{\it Claim.} $G_S$ has at most $2^{d/2-\ell}$ connected components.

\medskip

\noindent{\it Proof of claim.} 
Let $X$ be a connected component of $G_S$ and let $B=X\cap C_0$. For any $b\in B$, the vertices of $B$ at distance 3 from $b$ in $G_S$ are of the form $b+e_1+e_i+e_{\sigma(i)}$, where $2\le i\le d$ and $\sigma(i)$ is the index of the column
$h_1+h_i = \mathbf{1}+h_i$
in $H$. Write $v_i=e_1+e_i+e_{\sigma(i)}$ for any $2\le i \le d$. The observation above implies that $B$ is equal to $b+V$, where $V$ denotes the subgroup generated by the vectors $v_i$, $2\le i \le d $. Note that $v_i=v_{\sigma(i)}$ for any  $2\le i \le d$, but if we define $V'$ as a subset of $V$ containing exactly only one of $v_i$ and $v_{\sigma(i)}$ for any $2\le i \le d$, the vectors of $V'$ are linearly independent (each one has Hamming weight 3, and their supports are pairwise disjoint apart from the first coordinate which is common to the support of all vectors of $V'$). It follows that $V$ has dimension at least $|V'|=\tfrac{d+1}2$, and thus $|B|\ge 2^{(d+1)/2}$. Note that $|X|=(d+1)|B|$, since each vertex of $X\setminus B$ has exactly one neighbor in $B$, and thus $|X|\ge (d+1)2^{(d+1)/2}=2^{(d+1)/2+\ell}$. It follows that $G_S$ has at most \[\frac{2^d}{2^{(d+1)/2+\ell}}=2^{d-(d+1)/2-\ell}\le2^{d/2-\ell}\]  components, which concludes the proof of the claim.
\hfill $\blacksquare$

\medskip

Consider the graph $R$ obtained from $Q_d$ by contracting each component of $G_S$ into a single vertex. As $Q_d$ is connected, $R$ is also connected, so it contains a spanning tree $\tilde{T}$. Each edge $\tilde{e}\in E(\tilde{T})$ corresponds to at least one edge $e$ in $Q_d$ between two connected components of $G_S$, say $X_i$ and $X_j$. We observe that whenever there is such an edge $e$, there is actually at least one other edge $e'$ between $X_i$ and $X_j$ which is not incident to $e$. To see this, write $e=xy$, and assume $x\in C_s$ and $y\in C_t$ (note that $s,t\ge 2$, since otherwise $e$ would lie in $G_S$). Observe that by definition of $G_S$, we have $x'=x+e_s+e_t\in X_i$ and $y'=y+e_t+e_s\in X_j$. As $x$ and $y$ are adjacent in $Q_d$ and $x-y=x'-y'$, $x'$ and $y'$ are also adjacent, so we can set $e'=x'y'$.

For each edge $\tilde{e}\in E(\tilde{T})$, we consider the two edges $e$ and $e'$ in $Q_d$ defined above and add to $S$ one endpoint from each edge. Let $S'$ be the resulting vertex set, with $|S'|\le |S|+2\cdot 2^{d/2-\ell}$. Note that the subgraph $G_{S'}$ of $Q_d$ consisting of all edges incident to $S'$ is spanning and 2-connected (recall that each connected component of $G_S$ is 2-connected), and \[\tvc(G_{S'})\le |S'|\le 2\cdot \tfrac{2^d}{d+1}+2\cdot 2^{d/2-\ell}=2\cdot (2^{d-\ell}+2^{d/2-\ell}),\] so the theorem follows by taking $H=G_{S'}$.
\end{proof}

By Corollary~\ref{cor:tvc} we obtain the following as an immediate consequence.

\begin{corollary}\label{cor:hypercube}
Let $\ell\ge 2$ be an integer and let $d=2^\ell-1$. Then the protocol $\Pi_2$ for equality in the hypercube $Q_d$ has per-bit cost at most \[\frac{2^{d}}{d+1}+\frac{2^{d/2}}{d+1}=\opt(Q_d)+O\left(\sqrt{\opt(Q_d)}\right).\]
\end{corollary}

\subsubsection{Grids} 

We also obtain an asymptotically optimal protocol for equality in grids.

\begin{figure}[htb]
  \centering
    \includegraphics[scale=0.8]{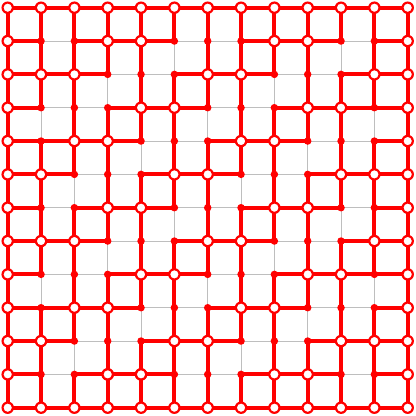}
  \caption{The vertex set $S$ (depicted as white circles) and the spanning 2-connected subgraph $G_S$ of the  square grid, in red.\label{fig:grid}}
\end{figure}

\begin{corollary}\label{cor:grids}
In the $n\times n$ square grid $G_n$, the protocol $\Pi_2$ of Theorem~\ref{thm:tvc2} has per-bit cost at most
$\tfrac15\,n^2+2n=(1+o(1))\opt(G_n)$.
\end{corollary}
\begin{proof}
In the $n\times n$ square grid $G=G_n$, consider the set $S$ of vertices $(i,j)\in [n]^2$ which either lie on the outerface, or are such that $i\in \{3j-2,3j-1\}\pmod 5$ (in words, in row $j$ of the grid, we add to $S$ all vertices located in columns whose index is $3j-2$ or $3j-1$ modulo 5). Let $G_S$ be the  subgraph of $G$ induced by the edges incident to $S$ (see Figure~\ref{fig:grid} for an illustration). Observe that $G_S$ is a spanning subgraph of $G$ and is 2-connected. By definition of $G_S$, it has a $S$ as a total vertex cover so  $\tvc(G_S)\le |S|\le \tfrac25\,n^2+4n$. By Corollary~\ref{cor:tvc}, the protocol $\Pi_2$ has per-bit cost at most $\tfrac12\,|S|\le \tfrac15\,n^2+2n$. By Lemma~\ref{lem:deg} we have $\opt(G)\ge \tfrac15\,n^2$, and the result follows.
\end{proof}

Similarly, we obtain asymptotically optimal protocols for equality in triangular grids and grids with all diagonals.

\begin{corollary}\label{cor:grids2}
In the $n\times n$ triangular grid $T_n$, the protocol $\Pi_2$ of Theorem~\ref{thm:tvc2} has per-bit cost at most
$\tfrac17\,n^2+2n=(1+o(1))\opt(T_n)$. In the $n\times n$ grid $P_n\boxtimes P_n$ with all diagonals (i.e., the strong product of two paths $P_n$), the protocol $\Pi_2$ of Theorem~\ref{thm:tvc2} has per-bit cost at most
$\tfrac19\,n^2+2n=(1+o(1))\opt(P_n\boxtimes P_n)$.
\end{corollary}

\begin{proof}
The proof is identical to that of Corollary~\ref{cor:grids}. We only need to find total vertex covers in some spanning 2-connected subgraph of the triangular grid (with $\tfrac27\,n^2+4n$ vertices), and of the grid with all diagonals (with  $\tfrac29\,n^2+4n$ vertices). Such sets are depicted in Figure~\ref{fig:grid2}.
\end{proof}

\begin{figure}[htb]
  \centering
    \includegraphics[scale=0.8]{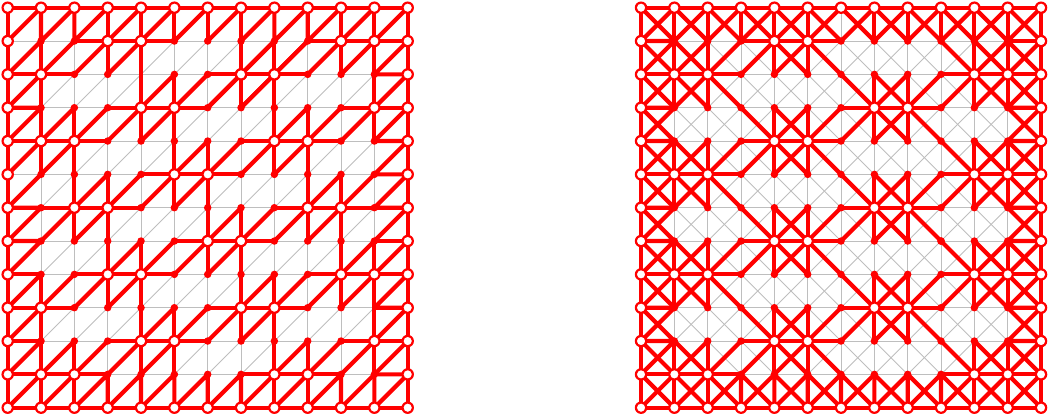}
  \caption{Total vertex covers (in white) of 2-connected spanning subgraphs (in red) of the triangular grid (left) and the grid with all diagonal (right).\label{fig:grid2}}
\end{figure}

We note that the lower order term in all these results can easily by improved by a factor of 2 (we have chosen not to do so for the sake of simplicity). It is not immediately clear how to reduce the lower order term to a constant.

\subsubsection{Regular graphs}

The following was proved in \cite{chen2017note}, under a different terminology.

\begin{theorem}[\cite{chen2017note}]\label{thm:cubictvc}
Every cubic graph on $n$ vertices has a total vertex cover on at most $3n/4$ vertices.
\end{theorem}

Combining Theorem~\ref{thm:cubictvc} and Theorem~\ref{thm:tvc2} for the upper bound and using Lemma~\ref{lem:deg} for the lower bound, we directly obtain the following $\tfrac32$-approximation for equality in cubic 2-edge-connected graphs.

\begin{corollary}\label{cor:3d}
For any $n$-vertex 2-edge-connected cubic graph $G$, the protocol $\Pi_2$ of Theorem~\ref{thm:tvc2} has per-bit cost at most $3n/8\le \tfrac32\,\opt(G)$.
\end{corollary}

\begin{remark}\label{rem:3drandom}
We can obtain a more efficient protocol for random cubic graphs as they are almost surely Hamiltonian \cite{robinson1994almost}, using the fact that $n$-cycles have a total vertex cover of size $(2/3+o(1))n$. This approach gives a protocol of per-bit cost $(1/3+o(1)) \,n\le (\tfrac43+o(1))\opt(G)$ for random cubic graphs, improving on Corollary~\ref{cor:3d} for almost all cubic graphs.
\end{remark}

\smallskip

Cubic graphs of large girth  share many properties with random cubic graphs (techniques solving problems in one class often work to solve problems in the other). We prove the following counterpart of the result above  for cubic graphs of large girth (we can actually use the same proof to give another proof of the $(1/3+o(1)) \,n$ result for random cubic graphs without using the Hamiltonicity of random cubic graphs, but we omit the details).

\begin{theorem}\label{th:3dgirth}
For any $\epsilon>0$, there is an integer $g$ such that the following holds. If $G$ is an $n$-vertex cubic graph which is  2-edge-connected and has girth at least $g$, then $G$ has a 2-connected spanning subgraph $G'$ with $\tvc(G')\le (\tfrac23+2\epsilon) n$, and thus the protocol $\Pi_2$ of Theorem~\ref{thm:tvc2} has per-bit cost $(\tfrac13+\epsilon) n\le (\tfrac43+\epsilon)\opt(G)$.
\end{theorem}

\begin{proof}
Let $M$ be a perfect matching and let $C_1, \ldots,C_\ell$ be the disjoint cycles in the subgraph of $G$ induced by the edges of $E(G)\setminus M$ (each cycle has length at least $g$, by definition). In each cycle $C_i$, there is a total vertex cover $S_i$ of size $\lceil 2|C_i|/3 \rceil\le (2/3+\epsilon)|C_i|$, for sufficiently large $g$. Let $H$ be the multigraph on $\ell$ vertices obtained from $G$ by contracting each cycle $C_i$ into a single vertex $v_i$. This graph is 2-connected, and by Lemma~\ref{lem:2co}, it  has a 2-connected spanning subgraph $H'$ with at most $2\ell-2\le 2\ell$ edges. Let $G'$ be the subgraph of $G$ consisting of the union of the cycles $C_i$ and the edges of $H'$. This graph $G'$ is a 2-connected spanning subgraph of $G$. For each edge of $H'$, add one of its endpoints in a new set $S_0$. Then $\bigcup_{i=0}^\ell S_i$ is a total vertex cover of $G'$ of size at most $(2/3+\epsilon)n+2\ell\le (2/3+\epsilon)n+2 n/g\le (2/3+2\epsilon)n$ for sufficiently large~$g$. By Corollary~\ref{cor:tvc}, we obtain  $\cost(G)\le \tfrac12\,\tvc(G')\le (1/3+\epsilon)n$, as desired.
\end{proof}

By extending the proof of Theorem~\ref{thm:cubictvc} we can prove that for $d\ge 3$,  every $d$-regular graph on $n$ vertices has a total vertex cover on at most $\tfrac{d}{d+1}\cdot n$ vertices. This is optimal for every $d$, since any total vertex cover of the complete graph $K_{d+1}$ (which is $d$-regular) contains at least $d$ vertices. It was proved in \cite{chen2017note} that there are two more extremal connected graphs in the case $d=3$, and all the other connected cubic graphs have a total vertex cover with less than $3n/4$ vertices. The result below implies that for $d\ge 5$, complete graphs are the only extremal examples. 
\begin{theorem}\label{thm:tvcregular2}
For every $d\ge 3$, every $d$-regular graph $G$ on $n$ vertices has a total vertex cover on at most $\frac{d}{d+1} \cdot n$ vertices. For $d\ge 5$, every connected $d$-regular graph $G$ on $n$ vertices distinct from $K_{d+1}$ has a total vertex cover on at most $\tfrac{d-\epsilon}{d+1}\cdot n$ vertices, with $\epsilon= \tfrac1{2d+1}$.
\end{theorem}
\begin{proof}
Let $S$ be a total vertex cover with the minimum number of vertices and let $T = V(G)\setminus S$.
For every $i\in \intv{0}{d - 1}$, let $S_i$ denote the subset of those vertices of $S$ that have exactly $i$ neighbors in $T$. Because $S$ is a total vertex cover and $G$ is $d$-regular, there is no vertex in $S$ with $d$ neighbors in $T$, hence $S = S_0 \cup \dots \cup S_{d-1}$. Also, since $S$ is a vertex cover, there are no edges between vertices in $T$.

The edges between $S$ and $T$ can be counted in two different ways, which yields the following equality:
\begin{equation}
    \sum_{i=0}^{d-1} i |S_i| = d |T|.
\end{equation}
As a consequence, we have:
\begin{align}
    |S| &= d |T| + |S_0| - \sum_{i=1}^{d-1} (i-1) |S_i|. \label{e:bnds}
\end{align}
Note that each vertex $u\in S_0$ has at least one neighbor in $S_{d-1}$, since otherwise we could remove $u$ from $S$ and still have a total vertex cover, which would contradict the minimality of $S$. On the other hand, each vertex of  $S_{d-1}$ has at most one neighbor in $S_0$, so it follows that $|S_0|\le |S_{d-1}|$. By \eqref{e:bnds}, it follows that 
\begin{align}
    |S| &\le  d |T|   - \sum_{i=1}^{d-2} (i-1) |S_i|-(d-3)|S_{d-1}| \label{e:2}
\end{align}

If $|T| \geq \frac{n}{d + 1}$, then $|S| = n-|T| \le  \frac{d}{d+1}\cdot n$. Otherwise $|T| \le \frac{n}{d + 1}$, and \eqref{e:2} implies that $|S|\le \tfrac{d}{d+1}\cdot n$, which proves the first part of the statement.

\medskip

We now assume for the remainder of the proof that $d\ge 5$, $G\neq K_{d+1}$, and that $G$ is connected. If $|T|\ge (\tfrac{1+\epsilon}{d+1})\cdot n$, then $|S|\le \tfrac{d-\epsilon}{d+1}\cdot n$, as desired. So we can assume that $|T|\le (\tfrac{1+\epsilon}{d+1})\cdot n$.

\medskip

\noindent {\it Claim.
For every $u\in T$, $N(u)\cap S_1$ induces a clique in $G$ and $u$ has a neighbor in $S_2 \cup \dots \cup S_{d - 1}$.}

\medskip

\noindent \emph{Proof of claim.}
To show the first part of the statement, suppose towards a contradiction that $u$ has two neighbors $v,w \in S_1$ that are not adjacent. Since $v,w\in S_1$, all the neighbors of $v$ and $w$ distinct from $u$ are in $S$. Then observe that $S\setminus \{vw\} \cup \{u\}$ is a total vertex cover, a contradiction to the minimality of $S$.
If $u$ has no neighbor in $S_2 \cup \dots \cup S_{d - 1}$, recall that it can neither have neighbors in $T$ nor in $S_0$, hence $N(u) \subseteq S_1$. By the first part of the statement, $G[\{u\} \cup N(u)]$ is a clique. As $G$ is $d$-regular and connected, $V(G) = \{u\} \cup N(u)$ and thus $G=K_{d+1}$, a contradiction.\hfill $\blacksquare$

\medskip

We now count the number of edges between $T$ and $S_2 \cup \dots \cup S_{d - 1}$. On the one hand, the claim above shows that this number is at least $|T|$. On the other hand, it is at most $\sum_{i=2}^{d-1}i|S_i|$. We thus obtain $|T|\le \sum_{i=2}^{d-1}i|S_i|$ and thus \[\frac{|T|}2\le \sum_{i=2}^{d-1}\frac{i}2\cdot |S_i|\le \sum_{i=2}^{d-2}(i-1)|S_i|+(d-3)|S_{d-1}|,\] where we have used that $d\ge 5$.
By \eqref{e:2}, it follows that we have \[|S|\le d|T|-|T|/2=(d-\tfrac12)|T|\le (d-\tfrac12)(\tfrac{1+\epsilon}{d+1})\cdot n=\tfrac{d-\epsilon}{d+1}\cdot n\] for $\epsilon= \tfrac1{2d+1}$, as desired.
\end{proof}

Theorem~\ref{thm:tvcregular2} can be combined with Theorem~\ref{thm:tvc2} to prove the following.

\begin{corollary}
For any $d\ge 3$ and  any 2-connected $d$-regular graph  $G$, the protocol $\Pi_2$ of Theorem~\ref{thm:tvc2} has per-bit cost at most $\frac{d}{2d+2}\cdot n = \tfrac{d}2\cdot \opt(G)$. If moreover, $d\ge 5$ and $G\ne K_{d+1}$, then the protocol $\Pi_2$ has per-bit cost at most $\tfrac{d-\epsilon}2 \cdot \opt(G)$, for some $\epsilon>0$ depending only on $d$.
\end{corollary}

Observe that complete graphs $K_{d+1}$ have a very efficient  protocol of per-bit cost 1, where a single vertex broadcasts its whole input. In combination with the corollary above, this outperforms the 4-approximation given by Theorem~\ref{thm:kv} whenever $d\le 8$.

\section{Conclusion}\label{sec:ccl}

\subsection{Monotone total vertex cover} Given a 2-connected graph $G$, let $\tvc^{\downarrow}(G)$ be the minimum of $\tvc(H)$ for all 2-connected spanning subgraphs $H$ of $G$. Equivalently, $\tvc^{\downarrow}(G)$ is the minimum cardinality of a subset $S$ of vertices of $G$ such that the subgraph $G_S$ of $G$ consisting of all edges incident to $S$ is spanning and 2-connected. Notice the similarity with the definition of $\wds(G)$, which is the minimum cardinality of a subset $S$ of vertices of $G$ such that the subgraph $G_S$ of $G$ consisting of all edges incident to $S$ is spanning and connected. A simple rephrasing of Corollary~\ref{cor:tvc} is that there is protocol for equality  in the local broadcast model in any 2-connected graph $G$ with per-bit cost  at most $\tfrac12  \tvc^{\downarrow}(G)$. Consequently, when $\tfrac12  \tvc^{\downarrow}(G)$ is close to $\tbn^*(G)$ we obtain almost-optimal  protocols for equality in $G$. This is the case for any graph that contains $K_{2,t}$ as a spanning subgraph (Section~\ref{sec:biclique}),  for hypercubes (Theorem~\ref{thm:hypercube}), and for grids (Corollary~\ref{cor:grids}). It was also used implicitly in Remark~\ref{rem:3drandom} that any $n$-vertex Hamiltonian graph $G$ satisfies $\tvc^{\downarrow}(G)\le \tfrac23(n+1)$, and in Theorem~\ref{th:3dgirth} to obtain an efficient protocol for  cubic graphs of large girth.

\smallskip
We believe that independently of its connections with multiparty equality,  this common variant of total vertex covers and weakly connected dominating sets is worth further investigation. The parameter $\tvc^{\downarrow}$ is in some sense much better behaved than $\tvc$: adding edges to $G$ does not increase $\tvc^{\downarrow}(G)$, while this might increase $\tvc(G)$.

\subsection{Open problems}

A natural problem is whether the fractional lower bound $\tbn^*(G)$ can always be attained by $\opt(G)$. We have not been able to find a graph for which the two parameters provably differ. A natural class to consider for potential counterexamples is the class of 2-edge-connected cubic graphs, which has a clean expression for~$\tbn^*$. For this class we have only been able to find protocols of per-bit cost $3n/8$, while the fractional lower bound is $n/4$. Note that if we have $\opt(G)=\tbn^*(G)$ for every graph $G$, then $\opt(G)$ can indeed be computed in polynomial time.

\smallskip

Another interesting problem is to bound the per-bit complexity of equality as a function of the number of vertices, regardless of any fractional lower bound. This question makes sense for the class of all graphs, but also for specific graph classes studied in structural and algorithmic graph theory, such as planar graphs or graphs of bounded treewidth.

\smallskip

We have seen that for various types of grids we could obtain asymptotically optimal protocols for equality, using total vertex covers in some 2-connected subgraphs, of size at most twice the fractional lower bound. The three types of grids we have considered can all be seen as Cayley graphs of $\mathbb{Z}^2$, with different generating sets. As the three constructions are very similar, a natural question is whether our constructions can be generalized to any Cayley graph of $\mathbb{Z}^2$.

\smallskip

It remains an open question whether the upper bound of $4\opt(G)$ on the cost of the protocol $\Pi_1$ in Theorem~\ref{thm:kv} can be improved.

\smallskip

Finally, in this paper we have only considered protocols that are deterministic and static. It is an interesting problem to understand whether significantly better bounds can be obtained for protocols that are randomized and/or non-static.

\subsection*{Acknowledgements} We thank Moritz Mühlenthaler and Alantha Newman for helpful discussions at the early stages of the project, and Zolt\'an Szigeti for pointing out references \cite{Gri21} and \cite{Whi32} to us. We also thank ChatGPT for  the suggested approaches for bounding the number of components of $G_S$ in the proof of Theorem~\ref{thm:hypercube}. 

\bibliography{breq} \bibliographystyle{alpha}
\end{document}